\documentclass[12pt,reqno]{article}

\usepackage[usenames]{color}
\usepackage{amssymb}
\usepackage{amsmath}
\usepackage{amsthm}
\usepackage{amsfonts}
\usepackage{amscd}
\usepackage{graphicx}
\usepackage{diagbox}
\usepackage{array}
\usepackage{thmtools}
\usepackage{thm-restate}
\usepackage{mathtools}
\DeclareMathOperator{\diam}{diam}
\usepackage[table]{xcolor}
\usepackage{booktabs}
\usepackage{listings}%for code
\newcommand{\pred}[1]{\cellcolor{blue!15}#1}
\usepackage[colorlinks=true,
linkcolor=webgreen,
filecolor=webbrown,
citecolor=webgreen]{hyperref}

\definecolor{webgreen}{rgb}{0,.5,0}
\definecolor{webbrown}{rgb}{.6,0,0}

\usepackage{color}
\usepackage{fullpage}
\usepackage{float}

\usepackage{graphics}
\usepackage{latexsym}
\usepackage{epsf}
\usepackage{breakurl}

\usepackage[margin=1in]{geometry} 
\usepackage{tikz}
\usetikzlibrary {arrows.meta}

\newcommand{\seqnum}[1]{\href{https://oeis.org/#1}{\rm \underline{#1}}}

\usepackage{rotating}
\newcommand{\cG}{{\cal G}}
\newcommand{\cH}{{\cal H}}
\newcommand{\cA}{{\cal A}}
\newcommand{\cL}{{\cal L}}
\newcommand{\cS}{{\cal S}}
\newcommand{\cI}{{\cal I}}

\newcommand{\ZZ}{\ensuremath{{\mathbb Z}}}

\begin{document}

\theoremstyle{plain}
\newtheorem{theorem}{Theorem}
\newtheorem{corollary}[theorem]{Corollary}
\newtheorem{lemma}[theorem]{Lemma}
\newtheorem{proposition}[theorem]{Proposition}

\theoremstyle{definition}
\newtheorem{definition}[theorem]{Definition}
\newtheorem{example}[theorem]{Example}
\newtheorem{conjecture}[theorem]{Conjecture}
\newtheorem{observation}[theorem]{Observation}
\newtheorem{question}[theorem]{Question}

\theoremstyle{remark}
\newtheorem{remark}[theorem]{Remark}

\begin{center}
\vskip 1cm{\bf\LARGE Optimal Diameters of High Multiplicity $g$-Golomb Rulers
}

\vskip 1 cm
\large
Aditya Gupta\\
Department of Mathematics\\
University of Washington\\
Seattle, WA 98195\\
USA\\
\href{mailto:adi3011@uw.edu}{\tt adi3011@uw.edu} \\
\ \\
Kevin O'Bryant\\
City University of New York\\
The Graduate Center and The College of Staten Island\\
2800 Victory Boulevard\\
Staten Island, NY 10314\\
USA \\
\href{mailto:kevin.obryant@csi.cuny.edu}{\tt kevin.obryant@csi.cuny.edu} \\

\end{center}

\vskip .2 in

\begin{abstract}
A set $\cG$ of $n$ integers is called a $g$-Golomb ruler if the difference between any two distinct elements of $\cG$ is repeated at most $g$ times. If $g=1$, these are also called $B_2$-sets, Sidon sets, and Babcock sets. We define $G(g,n)$ to be the minimum diameter of a $g$-Golomb ruler with $n$ marks. In this paper, we prove that for every $\epsilon>0$ and all sufficiently large positive integers $b$, if $g\ge(\pi/3+\epsilon)b^{3/2}$, then $G(g,g+b)=g+2b-2$. Explicit thresholds are given for $b\le251$. The main technique is through an arithmetic property of the integers that are \emph{not} in a $g$-Golomb ruler, leading us to introduce LM rulers, a class of rulers where every positive integer distance $d$ occurs as a difference at most $d-1$ times. Writing $L(n)$ for the minimum diameter of an $n$-element LM ruler, we prove $L(n)\ge\sqrt{8/9}\,(n-1)^{3/2}$ for every positive integer $n$ and, using a Farey construction, $L(n)\le(\pi/3+o(1))n^{3/2}$ as $n\to\infty$. We also provide computationally verified optimal LM rulers for every $1\le n\le250$.
\end{abstract}

\clearpage
\section{Introduction}\label{sec:intro}
A Golomb ruler is a set $\cG$ of integers with the property that for each $d\ge 1$, there is at most 1 pair $(a,b)\in\cG\times\cG$ with $d=b-a$. That is, with these marks on a ruler, there is at most one way to measure a distance of $d$. Golomb rulers are also called Sidon sets, Babcock sets, and $B_2$ sets, and there are several generalizations that have received attention~\cite{2004.Obryant}.

For example, $\{0,1,4,6\}$ is a Golomb ruler with four marks and diameter $6$. Its six positive differences are $1,2,3,4,5,6$, each occurring exactly once.

Here, we consider $g$-Golomb rulers, where $g$ is a positive integer. Some small examples were computed in \cite{1984.Atkinson&Hassenklover} and \cite{1986.Atkinson&Santoro&Urrutia}, and families of examples in \cite{2015.Caicedo&Martos&Trujillo} and \cite{2021.Ojeda&Urbano&Solarte}.
\begin{definition}\label{def:g-Golomb ruler}
    Let $g$ be a positive integer. A set $\cG$ of integers is a \emph{$g$-Golomb ruler} if, for every positive integer $d$, there are at most $g$ pairs $(a,b) \in \cG \times\cG$ with $d=b-a$. For each positive integer $n$, we set $G(g,n)$ to be the minimum diameter of a $g$-Golomb ruler with $n$ marks.
\end{definition}
The On-Line Encyclopedia of Integer Sequences (OEIS)~\cite{OEIS} is a database of integer sequences, each assigned an identifier beginning with A. OEIS Sequences~\seqnum{A003022}, \seqnum{A392461}, \seqnum{A392462}, \seqnum{A392463}, \seqnum{A395265} are $G(1,n)$, $G(2,n)$, $G(3,n)$, $G(4,n)$, $G(5,n)$, respectively. As part of this work, we have extended $G(g,n)$ for $2\le g \le 5$.

The set $\{1,2,\dots,n\}$ is a $g$-Golomb ruler for $g\ge n-1$, and so trivially $G(g,n)=n-1$ for $n\le g+1$. The table of non-trivial values---$G(g,g+b)$, where $g$ and $b$ are positive integers---is in Table~\ref{tab:sideways table} and in the OEIS as sequence \seqnum{A393905}. Row $g$ of Table~\ref{tab:sideways table} is asymptotically $g^{-1}n^2$~\cite{2015.Caicedo&Martos&Trujillo}. Our main result, Theorem~\ref{theorem:main}, identifies precisely the infinite ends of the columns of Table~\ref{tab:sideways table}.

\begin{table}[htbp]\centering\scriptsize\setlength{\tabcolsep}{2pt}\renewcommand{\arraystretch}{1.05}\resizebox{\textwidth}{!}{
\begin{tabular}{r@{\quad}|@{\quad}*{27}{c}}\toprule$g \backslash b$& 1 & 2 & 3 & 4 & 5 & 6 & 7 & 8 & 9 & 10& 11 & 12 & 13 & 14 & 15 & 16 & 17 & 18 & 19 & 20& 21 & 22 & 23 & 24 & 25 & 26 & 27 \\\midrule1  & \pred{1} & \pred{3} & 6 & 11 & 17 & 25 & 34 & 44 & 55 & 72 & 85 & 106 & 127 & 151 & 177 & 199 & 216 & 246 & 283 & 333 & 356 & 372 & 425 & 480 & 492 & 553 & 585 \\2  & \pred{2} & \pred{4} & 6 & 9 & 13 & 18 & 23 & 29 & 36 & 44 & 53 & 63 & 74 & 84 & 97 & 111  & 126  &  &  &  &  &  &  &  &  &  &  \\3  & \pred{3} & \pred{5} & \pred{7} & 10 & 13 & 16 & 20 & 25 & 30 & 35 & 42 & 49 & 56 & 64 & 73 &  &  &  &  &  &  &  &  &  &  &  &  \\4  & \pred{4} & \pred{6} & \pred{8} & 10 & 13 & 16 & 20 & 23 & 28 & 32 & 37 & 43 & 49 & 55 &  &  &  &  &  &  &  &  &  &  &  &  &  \\5  & \pred{5} & \pred{7} & \pred{9} & 11 & 14 & 16 & 20 & 23 & 27 & 31 & 35 & 40 & 45 & 50 &  &  &  &  &  &  &  &  &  &  &  &  &  \\6  & \pred{6} & \pred{8} & \pred{10} & \pred{12} & 14 & 17 & 20 & 23 & 26 & 30 & 34 & 38 &  &  &  &  &  &  &  &  &  &  &  &  &  &  &  \\7  & \pred{7} & \pred{9} & \pred{11} & \pred{13} & 15 & 18 & 20 & 23 & 26 & 30 & 33 & 37 &  &  &  &  &  &  &  &  &  &  &  &  &  &  &  \\8  & \pred{8} & \pred{10} & \pred{12} & \pred{14} & 16 & 18 & 21 & 24 & 27 & 30 & 33 & 37 &  &  &  &  &  &  &  &  &  &  &  &  &  &  &  \\9  & \pred{9} & \pred{11} & \pred{13} & \pred{15} & \pred{17} & 19 & 22 & 24 & 27 & 30 & 34 & 37 &  &  &  &  &  &  &  &  &  &  &  &  &  &  &  \\10 & \pred{10} & \pred{12} & \pred{14} & \pred{16} & \pred{18} & 20 & 23 & 25 & 28 & 31 & 34 & 37 &  &  &  &  &  &  &  &  &  &  &  &  &  &  &  \\11 & \pred{11} & \pred{13} & \pred{15} & \pred{17} & \pred{19} & 21 & 23 & 26 & 29 & 31 & 34 & 37 &  &  &  &  &  &  &  &  &  &  &  &  &  &  &  \\12 & \pred{12} & \pred{14} & \pred{16} & \pred{18} & \pred{20} & 22 & 24 & 27 & 29 & 32 & 35 & 38 &  &  &  &  &  &  &  &  &  &  &  &  &  &  &  \\13 & \pred{13} & \pred{15} & \pred{17} & \pred{19} & \pred{21} & \pred{23} & 25 & 28 & 30 & 33 & 35 & 38 &  &  &  &  &  &  &  &  &  &  &  &  &  &  &  \\14 & \pred{14} & \pred{16} & \pred{18} & \pred{20} & \pred{22} & \pred{24} & 26 & 28 & 31 & 33 & 36 & 39 &  &  &  &  &  &  &  &  &  &  &  &  &  &  &  \\15 & \pred{15} & \pred{17} & \pred{19} & \pred{21} & \pred{23} & \pred{25} & 27 & 29 & 32 & 34 & 37 & 40 &  &  &  &  &  &  &  &  &  &  &  &  &  &  &  \\16 & \pred{16} & \pred{18} & \pred{20} & \pred{22} & \pred{24} & \pred{26} & 28 & 30 & 33 & 35 & 38 & 40 &  &  &  &  &  &  &  &  &  &  &  &  &  &  &  \\17 & \pred{17} & \pred{19} & \pred{21} & \pred{23} & \pred{25} & \pred{27} & \pred{29} & 31 & 33 & 36 & 38 & 41 &  &  &  &  &  &  &  &  &  &  &  &  &  &  &  \\18 & \pred{18} & \pred{20} & \pred{22} & \pred{24} & \pred{26} & \pred{28} & \pred{30} & 32 & 34 & 37 & 39 & 42 &  &  &  &  &  &  &  &  &  &  &  &  &  &  &  \\19 & \pred{19} & \pred{21} & \pred{23} & \pred{25} & \pred{27} & \pred{29} & \pred{31} & 33 & 35 & 38 & 40 & 43 &  &  &  &  &  &  &  &  &  &  &  &  &  &  &  \\20 & \pred{20} & \pred{22} & \pred{24} & \pred{26} & \pred{28} & \pred{30} & \pred{32} & 34 & 36 & 38 & 41 & 43 &  &  &  &  &  &  &  &  &  &  &  &  &  &  &  \\21 & \pred{21} & \pred{23} & \pred{25} & \pred{27} & \pred{29} & \pred{31} & \pred{33} & \pred{35} & 37 & 39 & 42 & 44 &  &  &  &  &  &  &  &  &  &  &  &  &  &  &  \\22 & \pred{22} & \pred{24} & \pred{26} & \pred{28} & \pred{30} & \pred{32} & \pred{34} & \pred{36} & 38 & 40 & 43 & 45 &  &  &  &  &  &  &  &  &  &  &  &  &  &  &  \\23 & \pred{23} & \pred{25} & \pred{27} & \pred{29} & \pred{31} & \pred{33} & \pred{35} & \pred{37} & 39 & 41 & 44 & 46 &  &  &  &  &  &  &  &  &  &  &  &  &  &  &  \\24 & \pred{24} & \pred{26} & \pred{28} & \pred{30} & \pred{32} & \pred{34} & \pred{36} & \pred{38} & 40 & 42 & 44 & 47 &  &  &  &  &  &  &  &  &  &  &  &  &  &  &  \\25 & \pred{25} & \pred{27} & \pred{29} & \pred{31} & \pred{33} & \pred{35} & \pred{37} & \pred{39} & 41 & 43 & 45 & 48 &  &  &  &  &  &  &  &  &  &  &  &  &  &  &  \\26 & \pred{26} & \pred{28} & \pred{30} & \pred{32} & \pred{34} & \pred{36} & \pred{38} & \pred{40} & \pred{42} & 44 & 46 & 49 &  &  &  &  &  &  &  &  &  &  &  &  &  &  &  \\27 & \pred{27} & \pred{29} & \pred{31} & \pred{33} & \pred{35} & \pred{37} & \pred{39} & \pred{41} & \pred{43} & 45 & 47 & 50 &  &  &  &  &  &  &  &  &  &  &  &  &  &  &  \\28 & \pred{28} & \pred{30} & \pred{32} & \pred{34} & \pred{36} & \pred{38} & \pred{40} & \pred{42} & \pred{44} & 46 & 48 & 50 &  &  &  &  &  &  &  &  &  &  &  &  &  &  &  \\29 & \pred{29} & \pred{31} & \pred{33} & \pred{35} & \pred{37} & \pred{39} & \pred{41} & \pred{43} & \pred{45} & 47 & 49 & 51 &  &  &  &  &  &  &  &  &  &  &  &  &  &  &  \\30 & \pred{30} & \pred{32} & \pred{34} & \pred{36} & \pred{38} & \pred{40} & \pred{42} & \pred{44} & \pred{46} & 48 & 50 & 52 &  &  &  &  &  &  &  &  &  &  &  &  &  &  &  \\31 & \pred{31} & \pred{33} & \pred{35} & \pred{37} & \pred{39} & \pred{41} & \pred{43} & \pred{45} & \pred{47} & \pred{49} & 51 & 53 &  &  &  &  &  &  &  &  &  &  &  &  &  &  &  \\32 & \pred{32} & \pred{34} & \pred{36} & \pred{38} & \pred{40} & \pred{42} & \pred{44} & \pred{46} & \pred{48} & \pred{50} & 52 & 54 &  &  &  &  &  &  &  &  &  &  &  &  &  &  &  \\33 & \pred{33} & \pred{35} & \pred{37} & \pred{39} & \pred{41} & \pred{43} & \pred{45} & \pred{47} & \pred{49} & \pred{51} & 53 & 55 &  &  &  &  &  &  &  &  &  &  &  &  &  &  &  \\34 & \pred{34} & \pred{36} & \pred{38} & \pred{40} & \pred{42} & \pred{44} & \pred{46} & \pred{48} & \pred{50} & \pred{52} & 54 & 56 &  &  &  &  &  &  &  &  &  &  &  &  &  &  &  \\35 & \pred{35} & \pred{37} & \pred{39} & \pred{41} & \pred{43} & \pred{45} & \pred{47} & \pred{49} & \pred{51} & \pred{53} & 55 & 57 &  &  &  &  &  &  &  &  &  &  &  &  &  &  &  \\36 & \pred{36} & \pred{38} & \pred{40} & \pred{42} & \pred{44} & \pred{46} & \pred{48} & \pred{50} & \pred{52} & \pred{54} & \pred{56} & 58 &  &  &  &  &  &  &  &  &  &  &  &  &  &  &  \\37 & \pred{37} & \pred{39} & \pred{41} & \pred{43} & \pred{45} & \pred{47} & \pred{49} & \pred{51} & \pred{53} & \pred{55} & \pred{57} & 59 &  &  &  &  &  &  &  &  &  &  &  &  &  &  &  \\38 & \pred{38} & \pred{40} & \pred{42} & \pred{44} & \pred{46} & \pred{48} & \pred{50} & \pred{52} & \pred{54} & \pred{56} & \pred{58} & 60 &  &  &  &  &  &  &  &  &  &  &  &  &  &  &  \\39 & \pred{39} & \pred{41} & \pred{43} & \pred{45} & \pred{47} & \pred{49} & \pred{51} & \pred{53} & \pred{55} & \pred{57} & \pred{59} & 61 &  &  &  &  &  &  &  &  &  &  &  &  &  &  &  \\40 & \pred{40} & \pred{42} & \pred{44} & \pred{46} & \pred{48} & \pred{50} & \pred{52} & \pred{54} & \pred{56} & \pred{58} & \pred{60} & 62 &  &  &  &  &  &  &  &  &  &  &  &  &  &  &  \\41 & \pred{41} & \pred{43} & \pred{45} & \pred{47} & \pred{49} & \pred{51} & \pred{53} & \pred{55} & \pred{57} & \pred{59} & \pred{61} & \pred{63} &  &  &  &  &  &  &  &  &  &  &  &  &  &  &  \\42 & \pred{42} & \pred{44} & \pred{46} & \pred{48} & \pred{50} & \pred{52} & \pred{54} & \pred{56} & \pred{58} & \pred{60} & \pred{62} & \pred{64} &  &  &  &  &  &  &  &  &  &  &  &  &  &  &  \\43 & \pred{43} & \pred{45} & \pred{47} & \pred{49} & \pred{51} & \pred{53} & \pred{55} & \pred{57} & \pred{59} & \pred{61} & \pred{63} & \pred{65} &  &  &  &  &  &  &  &  &  &  &  &  &  &  &  \\44 & \pred{44} & \pred{46} & \pred{48} & \pred{50} & \pred{52} & \pred{54} & \pred{56} & \pred{58} & \pred{60} & \pred{62} & \pred{64} & \pred{66} &  &  &  &  &  &  &  &  &  &  &  &  &  &  &  \\45 & \pred{45} & \pred{47} & \pred{49} & \pred{51} & \pred{53} & \pred{55} & \pred{57} & \pred{59} & \pred{61} & \pred{63} & \pred{65} & \pred{67} &  &  &  &  &  &  &  &  &  &  &  &  &  &  &  \\\bottomrule\end{tabular}}
\caption{Known values of $G(g,g+b)$ for $1\le g\le45$ and $1\le b\le27$. Shaded entries use the computed thresholds $g\ge L(b-1)+1$ ($b\ge2$) from Subsection~\ref{subsec:LM computations}, together with $G(g,g+1)=g$ for $b=1$.}
\label{tab:sideways table}
\end{table}

In proving our values of $G(g,g+b)$, it is necessary to introduce Linear Multiplicity (LM) rulers.
\begin{definition}\label{def:LM ruler}
    A set $\cL$ of integers is an LM ruler if every positive integer $d$ has at most $d-1$ representations as a difference of elements of $\cL$. For each positive integer $n$, let $L(n)$ be the minimum possible diameter of an LM ruler with $n$ elements.
\end{definition}
We provide optimal LM rulers for every $1\le n\le250$. The supplementary file \texttt{n\_250.json} contains the witnesses; Table~\ref{tab:LM rulers} displays examples for $1\le n\le18$. Their lengths form sequence~\seqnum{A392517}.
\begin{theorem}
    Let $n$ be a positive integer. Suppose that $\{0=\ell_1<\ell_2<\dots<\ell_n\}$ is an LM ruler. Then \[\ell_n \ge \frac{2\sqrt{2}}{3} \cdot (n-1)^{3/2}.\]
    Moreover, as $n\to\infty$,
    \[
    L(n)\le\left(\frac\pi3+o(1)\right)n^{3/2}.
    \]
\end{theorem}

The connection between LM rulers and $g$-Golomb rulers is given in Theorem~\ref{theorem:main}, which is the main result of this work.
\begin{theorem}\label{theorem:main}
    Let $g$ and $b$ be positive integers. If $b=1$, then $G(g,g+1)=g$. If $b\ge2$ and $g\ge L(b-1)+1$, then $G(g,g+b)=g+2b-2$. In particular, for every $\epsilon>0$ and all sufficiently large $b$, this equality holds whenever
    \begin{equation}\label{eq:main asymptotic threshold}
    g\ge\left(\frac\pi3+\epsilon\right)b^{3/2}.
    \end{equation}
\end{theorem}

In Section~\ref{sec:lowerbound} we prove $G(g,g+b) \ge g +2b-2$. In Section~\ref{sec:exact constructions} we explore the specific cases of $b=1,2,3,4$. In Section \ref{sec:LM rulers} we develop an upper and lower bound for an LM ruler (defined below). We then use the results of Section \ref{sec:LM rulers} in Section~\ref{sec:upperbound} to prove that $G(g,g+b) \le g+2b-2$.

\section{Notation and Definitions}\label{sec:notation}

For any integers $a$ and $b$, define the discrete interval
\[ [a,b] \coloneqq \{k \in \mathbb{Z} : a \le k \le b\}. \]
The diameter of a set $\cA$ is given by
\[ \diam(\cA) \coloneqq \max \cA - \min \cA.\]

For a set $\cA \subset \ZZ$, the representation function $r_{\cA-\cA}(d)$ denotes the number of times a distance $d \in \mathbb{Z}^+$ occurs as a difference between distinct elements of $\cA$. Namely,
\[ r_{\cA-\cA}(d) \coloneqq \#\{(a_1, a_2) \in \cA \times \cA : a_2 - a_1 = d\}. \]
With this notation, we can restate Definition~\ref{def:g-Golomb ruler} as: a set $\cG$ is a $g$-Golomb ruler if for all $d\ge1$ we have $r_{\cG-\cG}(d)\le g$. We can restate Definition~\ref{def:LM ruler} as: a set $\cL$ is an LM ruler if for all $d\ge 1$ we have $r_{\cL-\cL}(d)\le d-1$.

We have found that to understand the elements of a $g$-Golomb ruler, it is essential to understand the elements that are missing from the ruler. 
\begin{definition}[Holes]
    The set $\cH(\cA)$ of \emph{holes} of a finite set $\cA$ of integers is the set of integers between the minimum and maximum of $\cA$ that are not in $\cA$. That is, \[\cH(\cA) \coloneqq \left[\min\cA,\max\cA\right] \setminus \cA.\]
\end{definition}
Specifically, the proof of Theorem~\ref{theorem:main} constructs minimum-diameter $g$-Golomb rulers by choosing a suitably placed LM ruler as the set of holes.

\section{A lower bound on the diameter of \texorpdfstring{$g$}{g}-Golomb rulers}
\label{sec:lowerbound}

\begin{lemma}\label{lem:min bound}
Let $g,b$ be positive integers, and let $\cG $ be a $g$-Golomb ruler with $g+b$ elements. Then 
\[ \diam(\cG) \ge g + 2b - 2. \]
\end{lemma}

\begin{proof}
Suppose that $\alpha\coloneqq \diam(\cG) \le g + 2b - 3$.
The discrete interval $[0,\alpha]$ contains exactly $\alpha + 1$ integers. Therefore, the exact number of holes $\cH=\cH(\cG)$ is
\[ |\cH| = (\alpha + 1) - |\cG| = (\alpha + 1) - (g + b). \]

We analyze the representation function for $d = 1$. In the complete interval $[0,\alpha]$, there are $\alpha$ pairs of distance 1. Each hole $h \in \cH $ removes at most two such pairs: $(h-1, h)$ and $(h, h+1)$. Therefore,
  \[g \ge r_{\cG -\cG }(1) 
    \ge \alpha - 2|\cH | 
    = \alpha - 2((\alpha + 1) - (g + b)), \]
which implies $\alpha \ge g+2b-2$.
\end{proof}

\section{Exact constructions of \texorpdfstring{$g$}{g}-Golomb rulers for small \texorpdfstring{$b$}{b}}\label{sec:exact constructions}

\begin{lemma}\label{lem:small_b}
Let $g$ and $b$ be positive integers. If $b\in\{1,2\}$ and $g\ge 1$, or $b=3$ and $g \ge3$, or $b=4$ and $g \ge 4$, then the minimum diameter of a $g$-Golomb ruler with $g+b$ marks is $G(g, g+b) = g + 2b - 2$.
\end{lemma}

\begin{proof}
Applying the general lower bound from Lemma \ref{lem:min bound}, we know that $G(g, g+b) \ge g + 2b - 2$. To show equality, we give minimal sets $\cG $ for each case (the shown sets are not always the unique sets with minimum diameter). It is straightforward to verify that these constructions satisfy $r_{\cG -\cG }(d) \le g$ for all valid distances and are therefore minimal $g$-Golomb rulers.

\[
\begin{array}{|r|cclc|}\hline
\text{valid $g$} &b & n & \text{$g$-Golomb ruler}& G(g,g+b) \\ \hline
g \ge 1 & 1 & g+1 & [0, g] & g \\
g \ge 1 & 2 & g+2 & [0, g] \cup \{g+2\} & g+2 \\
g \ge 3 & 3 & g+3 & [0, g-1] \cup \{g+1, g+2\} \cup \{g+4\} & g+4\\
g \ge 4& 4 & g+4 & [0, g-2] \cup \{g, g+1\} \cup \{g+3\} \cup \{g+5, g+6\} & g+6 \\ \hline
\end{array}
\]

We show these rulers visually in Figure~\ref{fig:visualizing small b}.
\end{proof}
\begin{figure}[H]
\begin{center}
\resizebox{0.95\textwidth}{!}{
\begin{tikzpicture}[
    scale=1.0,
    dot/.style={circle, fill=black, inner sep=0pt, minimum size=5pt},
    open/.style={circle, draw=black, fill=white, inner sep=0pt, minimum size=5pt},
    >=stealth
  ]

  \def\rowsep{-1.2}

  \def\y{0}
  \node[left] at (-0.5, \y) {$b=1$};
  \draw[->] (-0.3, \y) -- (13.0, \y);
  \fill[black!15] (0, \y-0.18) rectangle (5.5, \y+0.18);
  \draw[thick]    (0, \y-0.18) rectangle (5.5, \y+0.18);
  \node[dot] at (0,   \y) {};
  \node[dot] at (5.5, \y) {};
  \node[below=6pt] at (0,   \y) {$0$};
  \node[below=6pt] at (5.5, \y) {$g$};
  \node at (2.25, \y) {$\cdots$};

  \def\y{\rowsep}
  \node[left] at (-0.5, \y) {$b=2$};
  \draw[->] (-0.3, \y) -- (13.0, \y);
  \fill[black!15] (0, \y-0.18) rectangle (5.5, \y+0.18);
  \draw[thick]    (0, \y-0.18) rectangle (5.5, \y+0.18);
  \node[dot] at (0,   \y) {};
  \node[dot] at (5.5, \y) {};
  \node[below=6pt] at (0,   \y) {$0$};
  \node[below=6pt] at (5.5, \y) {$g$};
  \node at (2.25, \y) {$\cdots$};
  \node[open] at (6.5, \y) {};
  \node[below=6pt] at (6.5, \y) {$g{+}1$};
  \node[dot] at (7.5, \y) {};
  \node[below=6pt] at (7.5, \y) {$g{+}2$};

  \def\y{2*\rowsep}
  \node[left] at (-0.5, \y) {$b=3$};
  \draw[->] (-0.3, \y) -- (13.0, \y);
  \fill[black!15] (0, \y-0.18) rectangle (4.5, \y+0.18);
  \draw[thick]    (0, \y-0.18) rectangle (4.5, \y+0.18);
  \node[dot] at (0,   \y) {};
  \node[dot] at (4.5, \y) {};
  \node[below=6pt] at (0,   \y) {$0$};
  \node[below=6pt] at (4.5, \y) {$g{-}1$};
  \node at (1.75, \y) {$\cdots$};
  \node[open] at (5.5, \y) {};
  \node[below=6pt] at (5.5, \y) {$g$};
  \node[dot] at (6.5, \y) {};
  \node[dot] at (7.5, \y) {};
  \node[below=6pt] at (6.5, \y) {$g{+}1$};
  \node[below=6pt] at (7.5, \y) {$g{+}2$};
  \node[open] at (8.5, \y) {};
  \node[below=6pt] at (8.5, \y) {$g{+}3$};
  \node[dot] at (9.5, \y) {};
  \node[below=6pt] at (9.5, \y) {$g{+}4$};

  \def\y{3*\rowsep}
  \node[left] at (-0.5, \y) {$b=4$};
  \draw[->] (-0.3, \y) -- (13.0, \y);
  
  \fill[black!15] (0, \y-0.18) rectangle (3.5, \y+0.18);
  \draw[thick]    (0, \y-0.18) rectangle (3.5, \y+0.18);
  \node[dot] at (0,   \y) {};
  \node[dot] at (3.5, \y) {};
  \node[below=6pt] at (0,   \y) {$0$};
  \node[below=6pt] at (3.5, \y) {$g{-}2$};
  \node at (1.25, \y) {$\cdots$};

  \node[open] at (4.5, \y) {};
  \node[below=6pt] at (4.5, \y) {$g{-}1$};

  \node[dot] at (5.5, \y) {};
  \node[dot] at (6.5, \y) {};
  \node[below=6pt] at (5.5, \y) {$g$};
  \node[below=6pt] at (6.5, \y) {$g{+}1$};

  \node[open] at (7.5, \y) {};
  \node[below=6pt] at (7.5, \y) {$g{+}2$};

  \node[dot] at (8.5, \y) {};
  \node[below=6pt] at (8.5, \y) {$g{+}3$};

  \node[open] at (9.5, \y) {};
  \node[below=6pt] at (9.5, \y) {$g{+}4$};
  
  \node[dot] at (10.5,  \y) {};
  \node[dot] at (11.5, \y) {};
  \node[below=6pt] at (10.5,  \y) {$g{+}5$};
  \node[below=6pt] at (11.5, \y) {$g{+}6$};
\end{tikzpicture}}
\end{center}
\caption{Visualization of the optimal $g$-Golomb rulers with $g+b$ marks for $1\le b \le 4$.}
\label{fig:visualizing small b}
\end{figure}
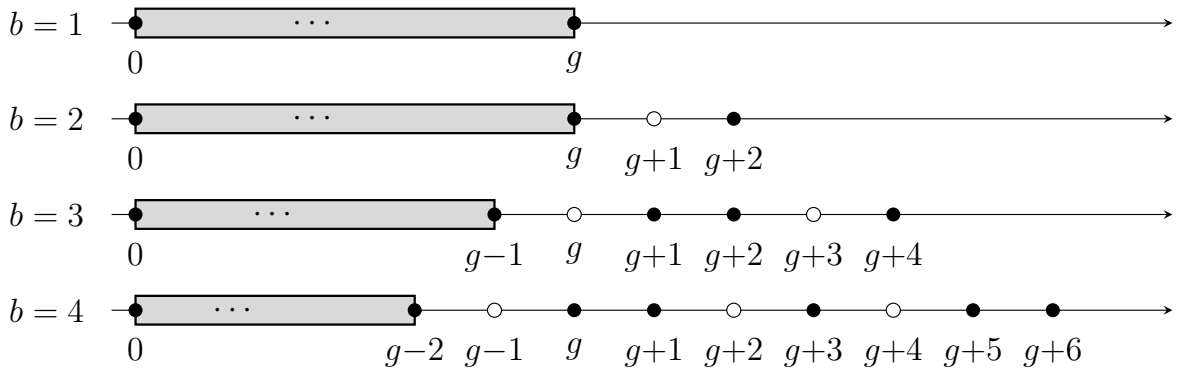

\section{Linear multiplicity (LM) rulers}
\label{sec:LM rulers}
We first establish the lower bound for LM rulers, then present the computed rulers, and finally give the Farey upper bound.

\subsection{A lower bound on the minimum diameter of an LM ruler}\label{sec:L lower bound}

While this lower bound is not needed for the proof of Theorem~\ref{theorem:main}, it demonstrates that the correct growth rate for LM rulers is $\Theta(n^{3/2})$.

\begin{theorem}\label{thm:lower_bound}
For any LM ruler $\cA$ of size $n \geq 1$, we have
    \[\diam(\cA) \geq \frac{2\sqrt{2}}{3}(n-1)^{3/2}.\]
\end{theorem}

\begin{proof}
Suppose that $\cA=\{\lambda_1<\dots<\lambda_n\}$, and set $s_i=\lambda_{i+1}-\lambda_i$ to be the first differences. Further, let 
    \[s_{(1)} \leq s_{(2)} \leq \cdots \leq s_{(n-1)}\]
be a reordering of the first differences $s_1,\dots,s_{n-1}$.

By the definition of LM ruler, there are at most $d-1$ differences of size $d$, and so there are at most
  \[\sum_{d=1}^\ell (d-1) = \frac{\ell(\ell-1)}2\]
first differences of size at most $\ell$.
Taking $\ell=\sqrt{2i}$, we see that there are at most $\sqrt{2i}(\sqrt{2i}-1)/2<i$ first differences of size $\sqrt{2i}$ or smaller, whence 
\begin{equation}\label{eq:greedy}
s_{(i)}>\sqrt{2i}.
\end{equation}

The diameter equals the sum of the first differences,
    \[\diam(\cA) = \sum_{i=1}^{n-1} s_i = \sum_{i=1}^{n-1} s_{(i)},\]
and applying Inequality~\eqref{eq:greedy} to each term gives
    \[\diam(\cA) \geq \sum_{i=1}^{n-1} \sqrt{2i} = \sqrt{2}\sum_{i=1}^{n-1}\sqrt{i}.\]
Since $\sqrt{t}$ is increasing, $\sqrt{i} \geq \int_{i-1}^{i}\sqrt{t}\,dt$ for each $i \geq 1$. Summing over $i$ gives
    \[\sum_{i=1}^{n-1}\sqrt{i} \;\geq\; \sum_{i=1}^{n-1}\int_{i-1}^{i}\sqrt{t}\,dt
    = \int_0^{n-1}\sqrt{t}\,dt = \frac{2}{3}(n-1)^{3/2}.\]
Therefore,
    \[\diam(\cA) \geq \sqrt{2} \cdot \frac{2}{3}(n-1)^{3/2},\]
as claimed.
\end{proof}

The lower bound also has an integer form: the smallest possible sum of the $n-1$ adjacent gaps uses one gap of size $2$, two of size $3$, and so on, stopping once $n-1$ gaps have been selected. This sum is a lower bound for $L(n)$. To attain it, these gaps must be ordered so that all positive differences satisfy the LM condition.

\subsection{Computed LM rulers}\label{subsec:LM computations}
Witnesses for $1\le n\le250$ are in the supplementary JSON file \texttt{n\_250.json}. Each attains the lower bound from Subsection~\ref{sec:L lower bound} and is therefore optimal. Table~\ref{tab:LM rulers} displays examples for $1\le n\le18$. The verifier \texttt{verify.py}, reproduced in Appendix~\ref{app:LM verification}, checks the LM condition and equality with the lower bound. Sequence~\seqnum{A392517} records the minimum diameters.
\begin{table}[p]
\centering
\begin{tabular}{|c|c|l|}
\hline
$n$ & $\mathrm{diam}(\cL )$ & Optimal LM ruler $\cS $ \\ \hline
1  & 0  & $\{0\}$ \\
2  & 2  & $\{0, 2\}$ \\
3  & 5  & $\{0, 2, 5\}$ \\
4  & 8  & $\{0, 2, 5, 8\}$ \\
5  & 12 & $\{0, 2, 5, 8, 12\}$ \\
6  & 16 & $\{0, 2, 5, 8, 12, 16\}$ \\
7  & 20 & $\{0, 2, 5, 8, 12, 16, 20\}$ \\
8  & 25 & $\{0, 2, 5, 8, 12, 16, 20, 25\}$ \\
9  & 30 & $\{0, 2, 5, 8, 12, 16, 20, 25, 30\}$ \\
10 & 35 & $\{0, 2, 5, 8, 12, 16, 20, 25, 30, 35\}$ \\
11 & 40 & $\{0, 2, 6, 9, 12, 16, 20, 25, 30, 35, 40\}$ \\
12 & 46 & $\{0, 2, 6, 9, 12, 16, 20, 25, 30, 35, 40, 46\}$ \\ 
13 & 52 & $\{0, 2, 6, 9, 12, 16, 20, 25, 30, 35, 40, 46, 52\}$\\
14 & 58 & $\{0, 2, 6, 9, 12, 16, 20, 25, 30, 35, 40, 46, 52, 58\}$\\
15 & 64 & $\{0, 2, 6, 9, 13, 16, 20, 25, 30, 35, 40, 46, 52, 58, 64\}$\\
16 & 70 & $\{0, 2, 7, 10, 14, 17, 21, 25, 30, 35, 40, 46, 52, 58, 64, 70\}$\\
17 & 77 & $\{0, 2, 7, 10, 14, 17, 21, 25, 30, 35, 40, 46, 52, 58, 64, 70, 77\}$\\
18 & 84 & $\{0, 2, 7, 10, 14, 17, 21, 25, 30, 35, 40, 46, 52, 58, 64, 70, 77, 84\}$\\
\hline
\end{tabular}
\caption{Examples from the computationally verified range $1\le n\le250$. Note that for $n\le 10$, the greedy set is optimal. For $n=15,16$ we thank Sean Irvine for sharing his computations with us and with the OEIS. For $n=17,18$ we thank Vrinda Inani for sharing her computations with us.}
\label{tab:LM rulers}
\end{table}

\subsection{A Farey construction for LM rulers}\label{subsec:farey upper bound}
In this section we give a construction of LM rulers using Farey Sequences. The construction gives
\[
L(n)\le\left(\frac\pi3+o(1)\right)n^{3/2}.
\]
For $Q\ge2$, define
\begin{align*}
\mathcal V_Q & \coloneqq \left\{(2^ra,2^rb):r\ge0,\ 1\le a<b,\ \gcd(a,b)=1,\ 2^rb\le Q\right\},\\
\mathcal N_Q & \coloneqq |V_Q|
\end{align*}
Order the vectors in $\mathcal V_Q$ by increasing slope $a/b$. Within a fixed slope, order them by increasing scalar $2^r$. Let
\[
(p_1,w_1),(p_2,w_2),\ldots,(p_{N_Q},w_{N_Q})
\]
be the resulting ordered list. Define
\[
x_0=0,\qquad x_j=w_1+w_2+\cdots+w_j
\]
for $1\le j\le N_Q$, and set
\begin{align*}
\cA_Q &\coloneqq\{x_0,x_1,\ldots,x_{N_Q}\}, \\
\mathcal D_Q &\coloneqq \diam(\cA_Q) 
\end{align*}

\begin{example}
For $Q=6$, the reduced pairs $(a,b)$ with $1\le a<b$, $\gcd(a,b)=1$, and $b\le6$ are
\[
(1,2),(1,3),(2,3),(1,4),(3,4),%\\
(1,5),(2,5),(3,5),(4,5),(1,6),(5,6).
\]
Including the scalar multiples gives $V_6$
\begin{center}
\begin{tabular}{c|c}
base vector & vectors used\\
$(1,2)$ & $(1,2),(2,4)$\\
$(1,3)$ & $(1,3),(2,6)$\\
$(2,3)$ & $(2,3),(4,6)$\\
$(1,4)$ & $(1,4)$\\
$(3,4)$ & $(3,4)$\\
$(1,5)$ & $(1,5)$\\
$(2,5)$ & $(2,5)$\\
$(3,5)$ & $(3,5)$\\
$(4,5)$ & $(4,5)$\\
$(1,6)$ & $(1,6)$\\
$(5,6)$ & $(5,6)$
\end{tabular}
\end{center}
from which we see $N_6 = 14$.
Ordering by increasing slope, and by increasing scalar within each fixed slope, gives
\[
%\begin{gathered}
(1,6),(1,5),(1,4),(1,3),(2,6),(2,5),(1,2),%\\
(2,4),(3,5),(2,3),(4,6),(3,4),(4,5),(5,6).
%\end{gathered}
\]
Thus the gap sequence is
\[
6,5,4,3,6,5,2,4,5,3,6,4,5,6,
\]
and the resulting ruler has $15$ marks:
\[
\cA_6=\{0,6,11,15,18,24,29,31,35,40,43,49,53,58,64\},
\]
and $D_6=64$. Thus, $L(15) \le 64$.
\end{example}

\begin{theorem}\label{thm:farey construction}
For every integer $Q\ge2$, the set $\cA_Q$ is an LM ruler. Thus,
\[
L(N_Q+1)\le D_Q.
\]
\end{theorem}

\begin{proof}
Every difference in $\cA_Q$ is a consecutive block sum of $w_1,\ldots,w_{N_Q}$. For a block $B=[i,j]$, define $V(B),P(B),W(B)$ by 
\[
V(B)\coloneqq (P(B),W(B)) \coloneqq \sum_{t=i}^j(p_t,w_t).
\]
Then $W(B)$ is the corresponding difference in $\cA_Q$.

Since $0<p_t<w_t$ for every $t$, every nonempty block satisfies
\[
0<P(B)<W(B).
\]
Thus, if $W(B)=d$, then
\[
P(B)\in\{1,2,\ldots,d-1\}.
\]
Therefore, for a fixed denominator sum $d$, there are only $d-1$ possible numerator sums. It remains to show that two different consecutive blocks cannot have the same pair of sums $(P(B),W(B))$.

Suppose two distinct nonempty consecutive blocks $B_1$ and $B_2$ satisfy
\[
V(B_1)=V(B_2).
\]
Neither block can contain the other, since all $w_t$ are positive. After removing $B_1\cap B_2$ and interchanging the names if necessary, we obtain two nonempty disjoint blocks $L$ and $R$, with every index in $L$ smaller than every index in $R$, such that
\[
\sum_{t\in L}(p_t,w_t)=\sum_{t\in R}(p_t,w_t).
\]
The slope of each vector sum is a weighted average of the slopes of the vectors being added up. Hence, since all slopes in $L$ are at most all slopes in $R$, the two sums can be equal only when every vector in $L\cup R$ has the same slope.

Thus any possible equality of vector sums must occur within one fixed slope class. Inside this slope class, the vectors have the form
\[
(2^ra,2^rb)
\]
with $a$ and $b$ constant. Taking $a,b$ common, a consecutive block in this class has multiplier sum
\[
2^u+2^{u+1}+\cdots+2^v=2^u(2^{v-u+1}-1).
\]
Trivially two different consecutive blocks in this class therefore cannot have the same vector sum.

\begingroup
\emergencystretch=2em
We have shown that two different consecutive blocks cannot have the same pair $(P(B),W(B))$. Now fix $d\ge1$. Each block with denominator sum $d$ gives a different numerator sum in
\par\endgroup
\[
\{1,2,\ldots,d-1\}.
\]
Hence there are at most $d-1$ such blocks, so
\[
r_{\cA_Q-\cA_Q}(d)\le d-1
\]
for every $d\ge1$. Thus $\cA_Q$ is an LM ruler.
\end{proof}

\begin{lemma}
For $Q\ge 2$,  
\(\displaystyle 
N_Q=\sum_{r\ge0}\ \sum_{2\le b\le\lfloor Q/2^r\rfloor}\varphi(b)\) and 
\(\displaystyle D_Q=\sum_{r\ge0}2^r\!\sum_{2\le b\le\lfloor Q/2^r\rfloor}b\varphi(b).
\)
\end{lemma}

\begin{proof}
Clearly $N_Q$ counts the vectors used in the construction, and hence the number of gaps. For each $r\ge0$ and each $2\le b\le\lfloor Q/2^r\rfloor$, there are $\varphi(b)$ choices of $a$ with $1\le a<b$ and $\gcd(a,b)=1$. Therefore
\[
N_Q=\sum_{r\ge0}\ \sum_{2\le b\le\lfloor Q/2^r\rfloor}\varphi(b).
\]
Similarly, the diameter is the sum of all gap sizes $2^rb$, so
\[
D_Q=\sum_{r\ge0}2^r\!\sum_{2\le b\le\lfloor Q/2^r\rfloor}b\varphi(b).
\]
\end{proof}

\begin{theorem}\label{thm:farey asymptotic}
We have
\[
\limsup_{n\to\infty}\frac{L(n)}{n^{3/2}}\le\frac\pi3.
\]
Equivalently,
\[
L(n)\le\left(\frac\pi3+o(1)\right)n^{3/2}.
\]
\end{theorem}

\begin{proof}
We use the standard estimate~\cite{DLMF}
\[
\sum_{b\le x}\varphi(b)\sim\frac3{\pi^2}x^2,
\]
and the summation by parts corollary
\[
\sum_{b\le x}b\varphi(b)\sim\frac2{\pi^2}x^3.
\]
Applying these estimates to the formulas for $N_Q$ and $D_Q$, and summing the geometric series, gives
\[
N_Q\sim\sum_{r\ge0}\frac3{\pi^2}\left(\frac Q{2^r}\right)^2=\frac4{\pi^2}Q^2
\]
and
\[
D_Q\sim\sum_{r\ge0}2^r\cdot\frac2{\pi^2}\left(\frac Q{2^r}\right)^3=\frac8{3\pi^2}Q^3.
\]
Therefore
\[
\frac{D_Q}{N_Q^{3/2}}\sim\frac{\frac8{3\pi^2}Q^3}{\left(\frac4{\pi^2}Q^2\right)^{3/2}}=\frac\pi3.
\]
For the special values $n_Q=N_Q+1$, Theorem~\ref{thm:farey construction} gives
\[
L(n_Q)\le D_Q=\left(\frac\pi3+o(1)\right)n_Q^{3/2}.
\]
Now let $n$ be arbitrary and choose $Q$ minimal such that
\[
N_Q+1\ge n.
\]
Taking an $n$-element subset of $\cA_Q$ gives
\[
L(n)\le D_Q.
\]
By minimality,
\[
N_{Q-1}+1<n\le N_Q+1.
\]
Since $N_Q\sim(4/\pi^2)Q^2$, this implies $n\sim N_Q$. Hence
\[
L(n)\le D_Q=\left(\frac\pi3+o(1)\right)N_Q^{3/2}=\left(\frac\pi3+o(1)\right)n^{3/2}.
\]
This proves the claim.
\end{proof}

\section{An upper bound on the diameter of \texorpdfstring{$g$}{g}-Golomb rulers}\label{sec:upperbound}

\begin{lemma}\label{lem:min distance}
Let $g$ and $b$ be positive integers with $g\ge b-1$. Let $\cI \coloneqq [0, g+2b-2]$, let $\cH \subseteq [b,\, g+b-1]$ with $|\cH| = b-1$, and let $\cG = \cI \setminus \cH$. Then for all integers $d \ge b$ we have $r_{\cG-\cG}(d) \le g$.
\end{lemma}

\begin{proof}
The setup can be visualized as follows.
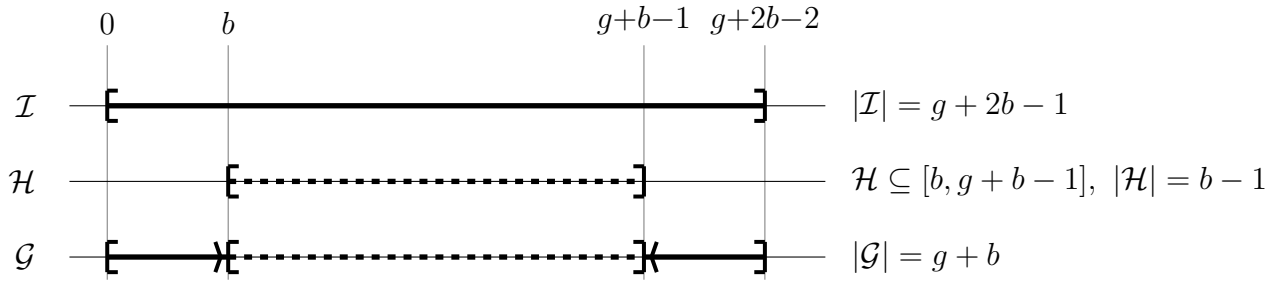
\begin{center}
\resizebox{\linewidth}{!}{%
\begin{tikzpicture}[
    scale=1.0,
    >=stealth
  ]

  \def\rowsep{-1.0}

  \def\xzero{0.0}
  \def\xbone{1.6}
  \def\xgpbmone{7.1}
  \def\xgp2bm2{8.7}

  \def\xaxisleft{-0.5}
  \def\xaxisright{9.5}
  \def\ytop{1.0}

  \def\bk{0.20}
  \def\sf{0.14}
  \def\pg{0.10}
  \def\po{0.07}

  \def\xcollabel{9.7}

  \foreach \xx/\lbl in {
      \xzero/{$0$},
      \xbone/{$b$},
      \xgpbmone/{$g{+}b{-}1$},
      \xgp2bm2/{$g{+}2b{-}2$}}{
    \draw[thin, gray] (\xx, \ytop-0.2) -- (\xx, 2*\rowsep-0.3);
    \node[above] at (\xx, \ytop-0.2) {\lbl};
  }

  \newcommand{\Lbrack}[2]{
    \draw[line width=1.5pt] (#1, #2-\bk) -- (#1, #2+\bk);
    \draw[line width=1.5pt] (#1, #2+\bk) -- (#1+\sf, #2+\bk);
    \draw[line width=1.5pt] (#1, #2-\bk) -- (#1+\sf, #2-\bk);
  }
  \newcommand{\Rbrack}[2]{
    \draw[line width=1.5pt] (#1, #2-\bk) -- (#1, #2+\bk);
    \draw[line width=1.5pt] (#1, #2+\bk) -- (#1-\sf, #2+\bk);
    \draw[line width=1.5pt] (#1, #2-\bk) -- (#1-\sf, #2-\bk);
  }
  \newcommand{\Lparen}[2]{
    \draw[line width=1.5pt] (#1+\po, #2+\bk) -- (#1, #2) -- (#1+\po, #2-\bk);
  }
  \newcommand{\Rparen}[2]{
    \draw[line width=1.5pt] (#1-\po, #2+\bk) -- (#1, #2) -- (#1-\po, #2-\bk);
  }

  \def\y{0}
  \node[left] at (-0.8, \y) {$\mathcal{I}$};
  \draw (\xaxisleft, \y) -- (\xaxisright, \y);
  \draw[line width=2pt] (\xzero, \y) -- (\xgp2bm2, \y);
  \Lbrack{\xzero}{\y}
  \Rbrack{\xgp2bm2}{\y}
  \node[right] at (\xcollabel, \y) {$|\mathcal{I}| = g+2b-1$};

  \def\y{\rowsep}
  \node[left] at (-0.8, \y) {$\mathcal{H}$};
  \draw (\xaxisleft, \y) -- (\xaxisright, \y);
  \draw[line width=2pt, dashed] (\xbone, \y) -- (\xgpbmone, \y);
  \Lbrack{\xbone}{\y}
  \Rbrack{\xgpbmone}{\y}
  \node[right] at (\xcollabel, \y) {$\mathcal{H}\subseteq [b,g+b-1],\ |\mathcal{H}|=b-1$};

  \def\y{2*\rowsep}
  \node[left] at (-0.8, \y) {$\mathcal{G}$};
  \draw (\xaxisleft, \y) -- (\xaxisright, \y);

  \draw[line width=2pt] (\xzero, \y) -- (\xbone, \y);
  \Lbrack{\xzero}{\y}
  \Rparen{\xbone-\pg}{\y}

  \draw[line width=2pt, dashed] (\xbone, \y) -- (\xgpbmone, \y);
  \Lbrack{\xbone}{\y}
  \Rbrack{\xgpbmone}{\y}

  \draw[line width=2pt] (\xgpbmone, \y) -- (\xgp2bm2, \y);
  \Lparen{\xgpbmone+\pg}{\y}
  \Rbrack{\xgp2bm2}{\y}
  \node[right] at (\xcollabel, \y) {$|\mathcal{G}| = g+b$};

\end{tikzpicture}}
\end{center}

If $d>g+2b-2=\diam(\cI)$, then no pair in $\cI$ has difference $d$, so $r_{\cG-\cG}(d)=0$.

Now suppose $b\le d\le g+2b-2$. In the full interval $\cI$, the number of pairs with difference $d$ is
\[
r_{\cI-\cI}(d)=g+2b-2-d+1=g+2b-1-d.
\]

If $2b-1 \le d$, then
\[
r_{\cG-\cG}(d)\le r_{\cI-\cI}(d)=g+2b-1-d\le g.
\]
It remains to consider $b\le d\le 2b-2$.

Every pair of elements of $\cG$ with difference $d$ is a pair of elements of $\cI$ with difference $d$ in which neither element of the pair is in $\cH$. Hence
\begin{align}
    r_{\cG-\cG}(d)
    &= g+2b-1-d
    -\#\left\{h \in [d,g+2b-2] : h\in\cH \text{ or } h-d \in \cH \right\} \label{eq:r(d) as holes}\\
    &\le g+2b-1-d - | \cH \cap [d,g+2b-2]|.\notag
\end{align}

If $d=b$, then $\cH\cap[d,g+2b-2]=\cH$, and therefore
\[
r_{\cG-\cG}(b)\le g+2b-1-b-(b-1)=g.
\]

Now suppose $b<d\le 2b-2$. Since $\cH\subseteq[b,g+b-1]$, the only elements of $\cH$ that may fail to lie in $[d,g+2b-2]$ are those in $[b,d-1]$. Thus
\[
|\cH\cap[d,g+2b-2]|
\ge |\cH|-|[b,d-1]|
=(b-1)-(d-b)=2b-1-d.
\]
Substituting this into \eqref{eq:r(d) as holes} gives
\[
r_{\cG-\cG}(d)
\le g+2b-1-d-(2b-1-d)=g.
\]
This proves the claim for all $d\ge b$.
\end{proof}

\begin{proof}[Proof of Theorem~\ref{theorem:main}.]
If $b=1$, the interval $[0,g]$ is a $g$-Golomb ruler with $g+1$ marks and diameter $g$. Lemma~\ref{lem:min bound} gives the matching lower bound, so $G(g,g+1)=g$. Henceforth, assume $b\ge2$.

By Lemma~\ref{lem:min bound}, we know that $G(g,g+b)\ge g+2b-2$. We need to show that, given $g,b$ with $g\ge L(b-1)+1$, there is a $g$-Golomb ruler $\cG$ with $g+b$ elements contained in $\cI \coloneqq [0,g+2b-2]$.

Our strategy is to take an LM ruler contained in $[b,g+b-1]$ with $b-1$ elements as $\cH$, and set $\cG\coloneqq \cI \setminus \cH$.

By Definition~\ref{def:LM ruler}, there is an LM ruler with $b-1$ elements and diameter $L(b-1)$. If $g\ge L(b-1)+1$, then we can take that ruler to be in the interval $[b,g+b-1]$, which has diameter $g-1$. We set this translated LM ruler as $\cH$, our set of holes. Then
\[
|\cG|=|\cI|-|\cH|=(g+2b-1)-(b-1)=g+b.
\]

By Lemma~\ref{lem:min distance}, the placement of $\cH$ in $[b,g+b-1]$ immediately gives
\[
r_{\cG-\cG}(d)\le g
\]
for all $d\ge b$, with no further conditions on $\cH$. We now focus on the distances $1\le d\le b-1$.

By Line~\eqref{eq:r(d) as holes} in the proof of Lemma~\ref{lem:min distance}, we have
\[
r_{\cG-\cG}(d)
=
g+2b-1-d
-
\#\left\{h \in [d,g+2b-2] : h\in\cH \text{ or } h-d \in \cH \right\}.
\]
Thus, to show that $r_{\cG-\cG}(d)\le g$, it is sufficient to show that
\begin{equation}\label{eq:to show}
\#\left\{h \in [d,g+2b-2] : h\in\cH \text{ or } h-d \in \cH \right\}
\ge 2b-1-d.
\end{equation}

Since $d\le b-1$ and $\cH\subseteq[b,g+b-1]$, both $\cH$ and $\cH+d\coloneqq \{h+d:h\in\cH\}$ are contained in $[d,g+2b-2]$. Therefore
\[
\left\{h \in [d,g+2b-2] : h\in\cH \text{ or } h-d \in \cH \right\}
=
\cH\cup(\cH+d).
\]
Hence
\begin{align*}
\#\left\{h \in [d,g+2b-2] : h\in\cH \text{ or } h-d \in \cH \right\}
&=|\cH\cup(\cH+d)|\\
&=|\cH|+|\cH+d|-|\cH\cap(\cH+d)|\\
&=2(b-1)-r_{\cH-\cH}(d).
\end{align*}
As $\cH$ is an LM ruler, we know that $r_{\cH-\cH}(d)\le d-1$, and so
\[
2(b-1)-r_{\cH-\cH}(d)\ge 2(b-1)-(d-1)=2b-1-d.
\]
This establishes Equation~\eqref{eq:to show}. Therefore $r_{\cG-\cG}(d)\le g$ for all $1\le d\le b-1$.

Thus $\cG$ is a $g$-Golomb ruler with $g+b$ elements and diameter $g+2b-2$. Combining this with Lemma~\ref{lem:min bound}, we obtain
\[
G(g,g+b)=g+2b-2.
\]

Finally, Theorem~\ref{thm:farey asymptotic} gives
\[
L(b-1)+1\le\left(\frac\pi3+o(1)\right)(b-1)^{3/2}+1
=\left(\frac\pi3+o(1)\right)b^{3/2}
\qquad(b\to\infty).
\]
Thus, for every $\epsilon>0$ and all sufficiently large $b$, condition~\eqref{eq:main asymptotic threshold} implies $g\ge L(b-1)+1$, proving the final assertion.
\end{proof}
\section{Open Problems}

\begin{enumerate}
    \item Fix $b$, find the minimum value of $\{G(g,g+b):g\ge 1\}$.
    \item Extend Table~\ref{tab:sideways table}. 
    \item Compute explicit upper and lower bounds for the missing entries of Table~\ref{tab:sideways table}.
    \item What is $g_0(b) \coloneqq \min \{ g: G(g,g+b)=g+2b-2\}$? The first $12$ terms of the sequence $(g_0(b))$ are computable from Table~\ref{tab:sideways table}: $1,1,2,4,6,8,11,14,17,20,24,28$. Is it true that $G(g_0(b)-1,g_0(b)-1+b)=g_0(b)+2b-2$?
    \item How many LM rulers are contained in $[0,n)$? How many, up to translation and reflection, have diameter exactly $L(n)$?
    \item Given $g,b$ what is the smallest $m$ such that $\ZZ/m\ZZ$ contains $g+b$ elements making a $g$-Golomb ruler. 
    \item Is $L(n)$ equal to the greedy bound for all $n$? This holds at least for $n\le 250$.
\end{enumerate}

\section*{Acknowledgments}
The first author is grateful to Professor Isabella Novik, whose initial homework problem sparked a lasting interest in the study of Golomb rulers. 

The first author thanks their parents, Ashish and Sonika Gupta, and their grandparents, whose early influence inspired the author to pursue the study of mathematics. Further, the first author wishes to thank Dr.\ Anuj Kumar for his unwavering support and inspiring the author to pursue research.

The authors also thank Theodore Meek and Pietro Monticone for their help with proofreading this paper. We also thank Sean A. Irvine and Vrinda Inani for sharing their computations. 

\bibliographystyle{jis}
\bibliography{golomb}

@techreport{1984.Atkinson&Hassenklover,
  author      = {M. D. Atkinson and A. Hassenklover},
  title       = {Sets of integers with distinct differences},
  institution = {School of Computer Science, Carleton University},
  address     = {Ottawa, Ontario, Canada},
  number      = {SCS-TR-63},
  year        = {1984},
}

@article{1986.Atkinson&Santoro&Urrutia,
  author  = {M. D. Atkinson and N. Santoro and J. Urrutia},
  title   = {Integer sets with distinct sums and differences and carrier frequency assignments for nonlinear repeaters},
  journal = {IEEE Trans. Commun.},
  volume  = {34},
  number  = {6},
  year    = {1986},
  pages   = {614--617},
  note    = {\url{https://doi.org/10.1109/TCOM.1986.1096587}},
}

@article{2015.Caicedo&Martos&Trujillo,
  author  = {Yadira Caicedo and Carlos A. Martos and Carlos A. Trujillo},
  title   = {{$g$}-{G}olomb rulers},
  journal = {Rev. Integr. Temas Mat.},
  volume  = {33},
  number  = {2},
  year    = {2015},
  pages   = {161--172},
  note    = {\url{https://doi.org/10.18273/revint.v33n2-2015006}},
}

@article{2004.Obryant,
  author  = {Kevin O'Bryant},
  title   = {A complete annotated bibliography of work related to {S}idon sequences},
  journal = {Electron. J. Combin.},
  volume  = {DS11},
  year    = {2004},
  pages   = {39 pp.},
  note    = {\url{https://doi.org/10.37236/32}},
}

@misc{DLMF,
  author       = {{NIST Digital Library of Mathematical Functions}},
  title        = {Release 1.2.4 of 2025-03-15},
  howpublished = {\url{https://dlmf.nist.gov/27.11.E6}},
  note         = {F. W. J. Olver, A. B. Olde Daalhuis, D. W. Lozier, B. I. Schneider, R. F. Boisvert, C. W. Clark, B. R. Miller, B. V. Saunders, H. S. Cohl, and M. A. McClain, eds., Eq.~27.11.6},
}

@misc{OEIS,
  author       = {N. J. A. Sloane and others},
  title        = {The {O}n-{L}ine {E}ncyclopedia of {I}nteger {S}equences},
  year         = {2026},
  howpublished = {Available at \url{https://oeis.org}},
}

@article{2021.Ojeda&Urbano&Solarte,
  author  = {Carlos Andres Martos Ojeda and David Fernando Daza Urbano and Carlos Alberto Trujillo Solarte},
  title   = {Near-optimal {$g$}-{G}olomb rulers},
  journal = {IEEE Access},
  volume  = {9},
  year    = {2021},
  pages   = {65482--65489},
  note    = {\url{https://doi.org/10.1109/ACCESS.2021.3075877}},
}

\clearpage
\appendix
\section{Verification of the computed LM rulers}
\label{app:LM verification}

This Python code checks the witness sets in \texttt{n\_250.json}. It verifies the LM condition for every set and confirms that each diameter equals the lower bound from Subsection~\ref{sec:L lower bound}.

\medskip
\noindent
\begin{lstlisting}[frame=single]
from collections import Counter
from pathlib import Path
import json

def lower_bound(n):
    total = 0
    remaining = n - 1
    gap = 2

    while remaining > 0:
        for _ in range(min(gap - 1, remaining)):
            total += gap
            remaining -= 1
        gap += 1
    return total

def verify():
    path = Path(__file__).with_name("n_250.json")
    rulers = json.loads(path.read_text(encoding="utf-8"))
    assert len(rulers) == 250
    for n, a in enumerate(rulers, 1):
        assert len(a) == n and all(type(x) is int for x in a)
        assert a == sorted(set(a))
        counts = Counter(a[j] - a[i]
            for i in range(n) for j in range(i + 1, n))
        assert all(c <= d - 1 for d, c in counts.items())
        assert a[-1] - a[0] == lower_bound(n)
    print("Verified optimal LM rulers for n=1..250.")

if __name__ == "__main__":
    verify()
\end{lstlisting}

\medskip
\hrule
\medskip

\noindent 2020 {\it Mathematics Subject Classification}:
Primary 05B10. Secondary 11B13.

\noindent \emph{Keywords:} 
Sidon Set, 
Generalized Golomb ruler, 
$g$-Golomb ruler

\medskip

\noindent (Concerned with sequences
\seqnum{A392517},
\seqnum{A003022},
\seqnum{A392461},
\seqnum{A392462}, 
\seqnum{A392463},
\seqnum{A393905},
\seqnum{A395265}.)

\medskip
\hrule
\medskip

\noindent
Return to \href{https://cs.uwaterloo.ca/journals/JIS/}{Journal of Integer Sequences home page}.
\vskip .1in

\end{document}